\documentclass[a4paper,12.5pt]{amsart}
\pdfoutput=1
\usepackage[margin=1.5in]{geometry}
\title{Strong multiplicity one for the Selberg class  }
\author{Michael Farmer}

\usepackage{bbm}

\usepackage{dirtytalk}
\usepackage{graphicx}
\usepackage{framed}
\usepackage{amsmath}
\usepackage{amssymb}
\usepackage{amsfonts}
\usepackage{cite}
\usepackage{mathrsfs}
\usepackage{array}
\usepackage{fullpage}
\usepackage [english]{babel}
\usepackage [autostyle, english = american]{csquotes}
\MakeOuterQuote{"}

\usepackage{cite}
\usepackage{amsthm} % package used to make the thoeorem environments work.
\usepackage{url}
\newtheorem{thm}{Theorem}[section]
\newtheorem{lem}[thm]{Lemma}

\theoremstyle{definition}
\newtheorem{defn}{Definition}[section]

\theoremstyle{remark}

\newcommand{\h}{h_{\mu}}
\newcommand{\g}{g_{\mu}}
\newcommand{\C}{C_{\mu}} 
\newcommand{\E}{E_{\mu}}

\newcommand{\Lb}{ \mathcal{L}}

\begin{document}

\maketitle

\begin{abstract}
    We study the problem of determining elements of the Selberg class by information on the coefficents of the Dirichlet series at the squares of primes, or information about the zeroes of the functions. 
\end{abstract}
\section{Introduction}
\subsection{The Selberg Class}
In \cite{Selberg}, Selberg axiomatised the expected properties of $L$ functions and introduced the "Selberg class". This is expected to coincide with the class of all arithmetically interesting $L$ functions. 
\begin{defn}
The Selberg class is the set of functions $F$ satisfying the following axioms:
\end{defn}
\begin{itemize}
    \item  In the half plane $\Re(s)>1$, the function $F(s)$ is given by a Dirichlet series $ \sum_{m=1}^{\infty} a_F(m) m^{-s}$, where $a_F(1)=1$  and $a_F(m) \ll_{\varepsilon} m^{\varepsilon}  $ for any $\varepsilon>0$. The estimate $a_F(m) \ll_{\varepsilon} m^{\varepsilon}  $ is known as the Ramunujan hypothesis. 
    \item  There exists a natural number $m_F$ such that $(s-1)^{m_F} F(s)$ extends to an analytic function of finite order in the entire complex plane. 
    \item There exists $Q_F>0$ , $\lambda_j(F)>0 $,  $\mu_j (F),  w_F \in \mathbb{C}$   with   $\Re(\mu_j(F))\geq 0$ and $|w_F|=1$,  such that the function $\Phi_F(s)= Q_F^s \Gamma_F(s) F(s)$, where 
    \[           \Gamma_F(s)= \prod_{j=1}^{N_F} \Gamma \left( \lambda_j (F)s + \mu_j(F)   \right) ,                \]
    satisfies the functional equation $ \Phi_F (s)= w_F \overline{ \Phi_F} (1-s) $. Here we use the notation that for any function $f$, $\overline{f} (s) = \overline{ f( \overline{s})}$. Recall that the degree of $F$ is defined to be the $ \sum_{j=1}^{N_f} 2 \lambda_j (F)$.
    \item   $F$ satisfies an Euler product formula. We can express $\log (F(s)) $ as a Dirichlet series 
    \[     \log (F(s))  = \sum_{m=2}^{\infty}     \frac{b_F(m) \Lambda(m)}{ m^s \log m}            \]
    where $b_F(m) \ll m^{\theta}$ for some $\theta<1/2$. We adopt the convention that $ b_F(m)$  is defined to be zero when $m$ is not a prime power. 
\end{itemize}
\subsection{Previous results}
It is generally believed that the Selberg class satisfies the "strong multiplicity one" principle. This states that if $F,G$ are two elements of the Selberg class with $a_F(p)=a_G(p)$ (equivalently $b_F(p)=b_G(p)$) for all but finitely many primes $p$, then $F=G$. This result is not known in this generality and historically more assumptions have been made to establish similar results. In \cite{Perelli},   Kaczorowski and  Perelli established this under an assumption on the coefficients at the squares of primes, namely that $ \left| a_F(p^2)- a_G(p^2) \right| $ is bounded on average, so for all large $x$, 
\[   \sum_{p \leq e^x} \frac{ \left| a_F(p^2)- a_G(p^2) \right| ^2 }{p} \log p \ll x  .     \]

In \cite{Sound}, Soundararajan  improved on this result by weakening the assumption to 
\[   \sum_{p \leq e^x} \frac{ \left| a_F(p^2)- a_G(p^2) \right| ^2 }{p} \log p \ll \exp \left(  \frac{x}{ \log x \left(  \log_2 x  \right)^5}     \right)   ,               \]
where we adopt the convention that $\log_2 x= \log \log x$.
In this paper, we weaken this assumption again by replacing the $ (\log_2 x)^{-5}$ term by $\varepsilon$.  We also establish a similar result assuming no  information on the coefficients at squares of primes, but instead use information on the zeroes of $F$ and $G$.  One could prove a more general result interpolating between the two cases, but we do not do this.  \\

\subsection{Statement of Results}
Write the (non-trivial) zeroes of $F$, $\rho_F$,  as $1/2+i \gamma_F$, where $\gamma_F \in \mathbb{C}$. Let $Z_F (T)$ denote the multiset of $\gamma_F \in \mathbb{R}$, with $|\gamma_F|\leq T$. Recall the symmetric difference notation for (multi)sets $A,B$,
\[       A  \bigtriangleup  B = ( A\backslash B) \cup (B\backslash A)        . \]
\begin{thm}
Suppose F and G are elements of the Selberg class with $a_F(p)=a_G(p)$  (equivalently $b_F(p)=b_G(p)$) for all $p \notin E $, where $E$ is a thin set of primes in the sense that
  \begin{equation}{\label{primes of E}}
      \# \{ p \in E: p\leq x \}  \ll x^{1/2 -\delta}
  \end{equation} 
  for some fixed $\delta>0$.  Suppose either

  \begin{equation}{\label{coinciding zeroes}}
   \left|   Z_F(T) \bigtriangleup Z_G(T) \right| = O \left(T  \frac{ \log T}{ \log_2 T }  \right) ,     \tag{i}
  \end{equation}
 or 
\begin{equation}{\label{Case 2}}
\sum_{p \leq e^x} \frac{ \left| a_F(p^2)- a_G(p^2) \right| ^2 }{p} \log p \ll_{\varepsilon} \exp \left( \frac{ \varepsilon x}{\log x}  \right), \; \; \; \text{for all} \; \;  \varepsilon>0. \tag{ii}
    \end{equation}               
   Then $F=G$.

\end{thm}
Our method of proof is similar to that of \cite{Sound}, but we can improve on the results by not choosing a smooth function in the explicit formula, but allowing the smoothness to be a parameter we can also control. This extra degree of freedom allows us to improve the result, but falls short of the conjectured result. \\

 The proof of both cases will be similar, but to make it more concise we shall define the following function. 
 \begin{equation*}
     \rho(T) =    \begin{cases}
      \frac{ \log T}{ \log_2 T },    & \text{if (i) holds } \\
      \log T, & \text{otherwise}. 
  \end{cases}      \end{equation*} 
  This will arise in the proof from the density of the zeroes of $F,G$. 
\section{preliminaries}{\label{background material}}
 
We shall need the following unconditional estimate which only relies on the Ramunujan hypothesis (first axiom for Selberg class). For any $\varepsilon>0$, we have 
\begin{equation}{\label{trivial bound}}
    \sum_{p \leq e^x} \frac{ \left| a_F(p^2)- a_G(p^2) \right| ^2 }{p} \log p \ll_{\varepsilon} e^{  \varepsilon x} .
\end{equation}
Let \begin{equation*}{\label{formula for A}}
    A_{\varepsilon} (x) = \begin{cases}
    \exp \left(  \frac{ \varepsilon x}{\log x}  \right),      & \text{if (ii) holds}  \\
    \exp \left( \varepsilon  x \right)  ,    & \text{otherwise}. \\
  \end{cases}      \end{equation*}

We will show the that (non-trivial) zeroes and poles of $F$ and $G$ coincide except on the critical line $\Re(s)=1/2$ when $F$ and $G$ satisfy (\ref{primes of E}). Write $c(n)= b_F(n)- b_G(n)$ and consider the function $ - F'/F (s)+ G'/G (s)$. For $\Re(s)>3/2$, the Euler product implies $F$ and $G$ have no zeroes , so the function is well defined in this region. For $\Re(s)>3/2$,
\begin{equation}{\label{log derivative}}
    - \frac{F'}{F} (s)+ \frac{G'}{G} (s) = \sum_{m=1}^{\infty} \frac{c(m) \Lambda(m)}{m^s}= \sum_{k=1}^{\infty} \sum_{p} \frac{ c(p^k) \log(p)}{p^{ks}} .
\end{equation}              
 For $k=1$, $c(p)=0$ unless $p \in E$. Fix arbitrary $\varepsilon>0$. Then for all  $s$ with $\Re(s)>1/2 - \delta+ 2 \varepsilon$ we have $ \Re(s)-\varepsilon> 1/2 - \delta+\varepsilon$. Using $c(p), \log p \ll_{\varepsilon} p^{\varepsilon}$ implies 
\[ \sum_{p \leq x} \frac{c(p) \log p}{p^s} \ll \sum_{p \leq x} \frac{\mathbbm{1}_{E} (p)   }{p^{\Re(s)- \varepsilon}}           .             \]
Summing by parts and using (\ref{primes of E}), the RHS converges uniformly as $x$ goes to infinity . Hence $\sum_{p} c(p) \log p/p^s$ is analytic for $\Re(s)>1/2- \delta$. Recall $c(m)\ll m^{\theta}$ for some fixed $\theta<1/2$. Fix an integer $M$ such that $ M (1/2 - \theta)>1$. Then 
\[  \sum_{k=2}^{\infty} \sum_{p} \frac{c(p^k) \log p }{p^{ks}} = \sum_{p} \left( \sum_{k=2}^{M-1} \frac{c(p^k) \log p }{p^{ks}} + \sum_{k=M}^{\infty} \frac{c(p^k) \log p }{p^{ks}}               \right)       \]
The first term is analytic for $\Re(s)>1/2$ since $c(p^k) \ll_{\varepsilon} p^{\varepsilon}$ for $ 2 \leq k \leq M-1$.   For  $\Re(s)>1/2$,
\begin{equation*}
    \begin{split}
    \sum_{p}\sum_{k=M}^{\infty} \frac{c(p^k) \log p }{p^{ks}}  \ll \sum_{p}
    \sum_{k=M}^{\infty} \frac{ \log p }{p^{k (\Re(s)-\theta)}} 
    \ll \sum_{p} \frac{\log p}{p^{M(\Re(s)- \theta)}} 
    \leq \sum_{p} \frac{\log p}{p^{M(\frac{1}{2}- \theta)} }=O(1)  
\end{split}
\end{equation*}
so  $\sum_{k=4}^{\infty} \sum_{p} c(p^k) \log p/ p^{ks}$ is analytic for $\Re(s)>1/2$. Putting this all together, the RHS of (\ref{log derivative}) is analytic for  $\Re(s)>1/2$ which allows us to analytically continue the LHS to $\Re(s)>1/2$. It follows that the zeroes and poles of $F$ and $G$ in this region agree (including multiplicities). Repeating the argument for $ \overline{F}(s)= \sum_{n=1}^{\infty} \overline{a_n} n^{-s} $ and $\overline{G}(s)$ shows that their poles and zeroes also coincide for $\Re(s)>1/2$. Putting this together with the functional equation shows that $\Phi_F$ and $\Phi_G$ have the same zeroes except for possibly on  $\Re(s)=1/2$ and poles of the same multiplicity at $s=1$.  \\ 

Let $T$ be a parameter which shall be taken to infinity. A standard application of the argument principle implies that the  (non-trivial) zeroes of $F$, $\rho_F$, satisfy the following estimate
\begin{equation}{\label{number of zeroes of F and G seperately}}
     \# \{ \rho_F : | \Im(\rho_F)  | \leq T                   \}= \frac{d_F}{\pi} T \log T + C_F T + O_F (\log T)
\end{equation}
where $C_F$ is a constant and $d_F$ is the degree of $F$.

 Let $\mu \geq 3$ be a  real parameter to be chosen later (it will be chosen so $ \mu \asymp \log\rho(T) $). For $x \in \mathbb{R}$, define 
\begin{equation*}
    \g (x)= \frac{2^{2\mu}}{\pi \sqrt{\mu} \binom{2\mu}{\mu}  }
    \begin{cases}
      (1-x^2)^{\mu-1/2}, & \text{if}\ |x|\leq 1\\
      0, & \text{otherwise}  
    \end{cases}
  \end{equation*}
  The function $\g$ has Fourier transform 
  \begin{equation*}
      \h (t)=  \frac{ \Gamma(\mu+1)}{ \sqrt{\mu} }  J_{\mu} ( |t|) \left( \frac{2}{ |t|} \right)^{\mu}
  \end{equation*}
  where $J_{\mu} (t)$ is the Bessel function of the first kind. This is a solution to the Bessel differential equation 
  \[ t^2 \frac{d^2 y}{dt^2}  + t \frac{dy}{dt} + (t^2-\mu^2) y =0    .      \] 
  
  By \cite{Bessel} we have  $ J_{\mu} (|t|) = O(1) $ uniformly for $\mu \geq 3$, $t \in \mathbb{R}$. Let us write $\C =  \Gamma(\mu+1) 2^{\mu} / \sqrt{\mu}$. We call the estimate 
  \begin{equation}{\label{asymp}}
      |\h (t)| \leq \C \left(  \frac{1}{|t|} \right)^{\mu}
  \end{equation}
  the asymptotic estimate for $\h$. 
  
  An important property of $\g$ we shall use repeatedly throughout the paper is the following bound. By Stirling's formula, for any $x \in \mathbb{R}$,
  \begin{equation}{\label{bound for g_n}}
      \g (x) \leq \g(0) \asymp 1 .
  \end{equation}
  Also, we can recover $\g$ from $\h$ by means of the Fourier inversion formula 
  \[      \g(x)= \frac{1}{2 \pi} \int_{-\infty}^{\infty} \h (t) e^{ -ix t} \, dt .             \]
  \\
Consider  a $C^2$ function $u$ which is compactly supported with Fourier transform $v$.  Recall we write the (non-trivial) zeroes of $F$, $\rho_F$ as $ 1/2 + i \gamma_F$. Approximating $C^2$ functions by smooth functions we can use \cite[Proposition 2.1]{Rand} to get the following explicit formula which relates  the (non-trivial) zeroes of $F$ to $b_F(m) \Lambda(m)$ at prime powers.
  \begin{equation}{\label{explicit formula original}}
    \begin{split}
        \sum_{\gamma_F} v( \gamma_F) &= m_f \left(  v \left( \frac{-i}{2} \right) + v  \left( \frac{i}{2}  \right)
    \right) \\
    &+ \frac{1}{2 \pi} \int_{-\infty}^{\infty} v(r) \left( 2 \log Q_F + \frac{\Gamma'_F}{\Gamma_F} \left( \frac{1}{2}+ir \right) +  \frac{\overline{\Gamma'_F}}{\overline{\Gamma_F}} \left( \frac{1}{2}- ir \right)     \right) dr  \\
    & - \sum_{m=1}^{\infty} \left( \frac{b_F(m) \Lambda(m)}{ \sqrt{m}} u \left( \log m \right)  + \frac{\overline{b_F(m)} \Lambda(m)}{ \sqrt{m}} u \left(- \log m \right)   \right),
     \end{split}
  \end{equation}
  where we sum over non-trivial zeroes of $F$ with multiplicity. \\
   Let $t \in [T,2 T]$. Let $L$ be another parameter depending on $T$ (which will chosen so  $L \asymp \rho(T) \log \rho(T)$. Notice that $\g$ is not a smooth function, but for $\mu \geq 3$, $\g$ is $C^2$. Hence by considering $v(r)= \h ( L (r-t))$ and using the Fourier inversion formula we deduce
  \begin{equation}{\label{shifted explicit formula}}
       \begin{split}
        \sum_{\gamma_F} \h( L(\gamma_f-t)) ) &= m_f \left(  \h \left(  L \left( \frac{-i}{2} -t \right) \right) + \h \left( L\left( \frac{i}{2} +t  \right) \right) \right) \\
    &+ \frac{1}{2 \pi} \int_{-\infty}^{\infty} \h ( L(r-t))  \left( 2 \log Q_F + \frac{\Gamma'_F}{\Gamma_F} \left( \frac{1}{2}+ir \right) + \frac{\overline{\Gamma'_F}}{\overline{\Gamma_F}} \left( \frac{1}{2}- ir \right)     \right) dr  \\
    & - \frac{1}{L} \sum_{m=1}^{\infty} \left( \frac{b_F(m) \Lambda(m)}{ m^{1/2 +it} } \g \left( \frac{\log m}{L}   \right)  + \frac{\overline{b_F(m)} \Lambda(m)}{ m^{1/2 -it}} \g \left(  \frac{- \log m}{L} \right) \right).
     \end{split}
        \end{equation}
 Denote the middle and third terms by  $ H_F (t,L,\mu), D_F(t,L.\mu)$  respectively. We shall use the following lemma to be  proven at the end. 
 \begin{lem}{\label{lem for H}}
  For our choice of $\mu$ and $L$, 
  \begin{equation}{\label{formula for H}} 
  H_F (t,L,\mu)= \g(0)  \frac{d_F \log T}{L} +  \frac{ O \left(  E_{\mu} (T) \right) }{L}  
 \end{equation}
 where 
 \[     E_{\mu} (T) =              \frac{C_{\mu}}{T}    + 1    \; \; \; \text{and} \; \;  \; C_{\mu}= \Gamma(\mu+1) 2^{\mu}/\sqrt{\mu} .                 \]
 \end{lem}

 \section{ Proof of theorem}
By our observations on the zeroes of $F,G$ above,  we have
\begin{equation}{\label{relationship}}
     Z_F (t,L,\mu)- Z_G (t,L,\mu) =    H_F (t,L,\mu)-    H_G  (t,L,\mu) - D_F(t,L,\mu) +    D_G(t,L,\mu)
 \end{equation}                      
  where  \[ Z_F(t,L,\mu) =   \sum_{\gamma_F \in \mathbb{R}} \h( L(\gamma_f-t)) )   .  \]
  We now want to record an estimate for $D_F(t,L, \mu)- D_G (t,L, \mu)$.  Choose $\varepsilon>0$ such that $1/2 - \varepsilon> 1/2 - \delta$, where $\delta$ is from (\ref{primes of E}). Since $  c(p)  \ll_{\varepsilon} p^{\varepsilon}$ and using (\ref{bound for g_n}), (for the rest of the paper we shall do this without mention), we have 
  \begin{equation*}
  \begin{split}
      \sum_{p} \frac{\left| c(p)    \right| \g(\log p/L) \log p} { \sqrt{p}}&\ll     \sum_{p \in E} \frac{1}{p^{1/2- \varepsilon}} 
      \ll 1 .
     \end{split}
  \end{equation*} 
  Since $ c(m) \ll m^{\theta}$ we see that 
  \[    \sum_{k \geq 3}  \sum_{ p } \frac{\left|c(p^k)  \right| \g(k \log p/L) \log p  }{p^{k/2}} \ll 1    .      \]
  By considering the logarithmic derivative of $F$, it is easy to see show that $b_F(p^2)=a_F(p^2) - a_F(p)b_F(p)/2$. Noticing that again by (\ref{primes of E})
  \[     \sum_{p} \frac{\left| a_F(p) b_F(p)- a_G(p) b_G(p)    \right| \g(2\log p/L) \log p }{p}\ll  1   ,              \]
 we have that
  \begin{equation*}
  \begin{split}
  L( D_F(t,L,\mu)- D_G(t,L,\mu))\ll 1 &+ \left|   \sum_{p} \frac{ a_F(p^2) - a_G(p^2)}{p^{1+2it}} \g\left(\frac{2\log p}{L} \right) \log p   \right| \\
  &+  \left|   \sum_{p} \frac{\overline{ a_F(p^2)} - \overline{a_G(p^2)}}{p^{1-2it}} \g\left(\frac{-2\log p}{L} \right) \log p   \right|   . 
  \end{split}
       \end{equation*}
       Using a mean-value estimate from Montgomery and Vaughan \cite[Corollary 3]{mean}, and using the fact $\g$ is compactly supported on $[-1,1]$, we have that 
       \begin{equation*}
       \begin{split}
         \int_{T}^{2T}  \left|   \sum_{p} \frac{ a_F(p^2) - a_G(p^2)}{p^{1+2it}} \g\left(\frac{2\log p}{L} \right) \log p   \right|^2 dt& \ll  \sum_{p \leq e^{L/2}} \frac{  \left| a_F(p^2) - a_G(p^2)   \right|^2}{p^2} (T+p) \log^2p \\
         &\ll    T + \sum_{p \leq e^{L/2}}   \frac{  \left| a_F(p^2) - a_G(p^2)   \right|^2}{p}  \log^2p           \\
         &\ll  T + O_{\epsilon} \left(\frac{L}{2} A_{\varepsilon} (L/2) \right)   , 
       \end{split}
         \end{equation*}
         where the last line comes from (\ref{trivial bound}). Dealing with the cross terms using Cauchy's inequality, we deduce that 
         \begin{equation}{\label{formula for D}}
             \int_{T}^{2T} \left( L( D_F(t,L,\mu)- D_G(t,L,\mu))  \right)^2 dt  \ll   T + O_{\varepsilon} \left(  A_{\varepsilon} (L)  \right).
         \end{equation}
         \\
 Let $W\geq 1$ be a large constant to be chosen later. Let $\Lb=\Lb(W)$ denote the set of $ t \in [ T, 2T]$ such that there exists $ \gamma \in Z_F(\infty) \bigtriangleup Z_G(\infty)$  in  $  \left( t- \frac{1}{W \rho(T)},   t+ \frac{1}{W\rho(T)}    \right)    $. Notice in either case (i) or (ii),   (i) or (\ref{number of zeroes of F and G seperately})  implies  for large enough $T$,  $\text{Meas} (\Lb)\ll T \rho(2T+1)/W\rho(T) \ll T/W$. Let $\widetilde{\Lb}$ denote $[T,2T] \backslash \Lb$. Then 
 \begin{equation*} \label{eq1}
\begin{split}
& \; \; \; \; \; L \int_{t \in \widetilde{\Lb}} \left| Z_F(t,L,\mu)- Z_G(t,L,\mu) \right| \, dt \\
&= L \int_{T}^{2T} \Bigg| \sum_{\substack{   {\gamma_F \in \mathbb{R} } \\ 
{|\gamma_F-t| \geq 1/W\rho(T)} }} \h( L (\gamma_F-t)) 
 -  \sum_{\substack{  \gamma_G \in \mathbb{R} \\   
{|\gamma_G-t| \geq 1/W\rho(T)} }} \h( L (\gamma_G-t))  \Bigg| \, dt \\
&=  L \int_{T}^{2T} \left| \sum_{\substack{    
{|\gamma-t| \geq 1/W\rho(T)}\\ 
 { \gamma \in Z_F(\infty) \bigtriangleup Z_G(\infty)}} }      
  \pm \h( L (\gamma-t))    \right| \, dt. 
\end{split}
\end{equation*}
We deal now with the contribution of zeroes outside $(T-1, 2T+1)$.  If the distance of $\gamma$ from $[T,2T]$ exceeds $m$ then the contribution from any term on the RHS of the above equation is \\ $ \leq \mu^{-1/2} \Gamma(\mu+1) \left( 2/  Lm    \right)^{\mu}$ by (\ref{asymp}). Since there are $\ll \log(T(m+1))$ ordinates $\gamma$ whose distance is between $m$ and $m+1$ by (\ref{number of zeroes of F and G seperately}), the total contribution is
 $\ll \mu^{-1/2} \Gamma(\mu +1) \left(  2/  L   \right)^{\mu} \log(T)$. This is because $\sum_{m=1}^{\infty} m^{-\mu} \log(m+1)$ and $ \sum_{m=1}^{\infty} m^{-\mu} $ are $ O(1)$ independent of $\mu \geq 3$.  This means 

\begin{equation} \label{estimate on z 1}
\begin{split}
& \; \; \; \; \;L \int_{t \in \widetilde{\Lb}} \left| Z_F(t,L,\mu)- Z_G(t,L,\mu) \right| \, dt \\
& \ll \mu^{-1/2} \Gamma(\mu +1) \left(  2/ L   \right)^ {\mu}\log(T)  +    \sum_{\substack{   {\gamma \in (T-1,2T+1) } \\
 { \gamma \in Z_F(2T+1) \bigtriangleup Z_G(2T+1)}} } \int_{|y| \geq L/W \rho(T)} \left|       
 \h( y)    \right| \,  dy \\
  &\ll  \mu^{-1/2} \Gamma(\mu +1) \left(  2/ L   \right)^ {\mu}\log(T)  +   T \rho(T)  \int_{|y| \geq L/W \rho(T)} \left|       
 \h( y)    \right| \, dy, 
\end{split}
\end{equation}
where the last line comes from (\ref{number of zeroes of F and G seperately}).  \\ 

Suppose now that $F\neq G$, so there is an integer $m$ such that $c(m) \neq 0$. Using (\ref{estimate on z 1}) we see that 
\begin{equation}{\label{estimate on z 2}}
\begin{split}
   & \; \; \; \; \;L \int_{t \in \tilde{\Lb}} m^{it} \left(Z_F(t,L,\mu)- Z_G(t,L,\mu) \right) \,  dt \\
   &\ll \mu^{-1/2} \Gamma(\mu +1) \left(  2/ L   \right)^ {\mu}\log(T)   +   T \rho(T)  \int_{|y| \geq L/W \rho(T)} \left|       
 \h( y)    \right| \,  dy .
 \end{split}
\end{equation}
However using (\ref{relationship}) this is also equal to 
\begin{equation}{\label{H,Z,D}}
     L      \int_{t \in \tilde{\Lb}} m^{it}          \left( H_F (t,L,\mu)-    H_G  (t,L,\mu) - D_F(t,L,\mu) +    D_G(t,L,\mu)         \right) \, dt .
\end{equation}             
From \cite{Sound}, (\ref{primes of E}) tells us that $d_F= d_G$. Recall that for large $T$, $ \text{Meas}(\Lb)\ll T/W   $.  So for large $T$ we have 
\[ L \int_{t \in \Lb }    \left| H_F(t,L,\mu)  - H_G(t,L,\mu)   \right| \, dt \ll E_{\mu} (T) \text{Meas}(\Lb) \ll  E_{\mu} (T) \frac{T}{W}  .      \]
We will also need the following lemma proved at the end of the paper.
\begin{lem}{\label{lem for derivative term}}
Let $F,G$ lie in the Selberg class and satisfy (\ref{primes of E}). For our choice of $\mu, L$, we have the following estimate, 
 \begin{equation*}\begin{split}
     \int_{T}^{2T} L m^{it} \left( H_F (t,L,\mu) - H_G (t,L,\mu) \right) \, dt \ll T\left(  \mu^{-1/4} + \frac{\C}{T} \right) + E_{\mu} (2T) +E_{\mu} (T) .
 \end{split}
      \end{equation*}  
\end{lem}
Putting our above formula together with the lemma we have   
\begin{equation}{\label{12b}}
    L      \int_{t \in \tilde{\Lb}} m^{it}          \left( H_F (t,L,\mu)-    H_G  (t,L,\mu)       \right) \, dt \ll  T\left(  \mu^{-1/4} + \frac{\C}{T} \right) + E_{\mu} (2T) +E_{\mu} (T) \frac{T}{W}.
\end{equation}
By Cauchy's inequality and (\ref{formula for D}) we have that 
\[  \int_{t \in \Lb }  L    \left| D_F (t,L,\mu)-    D_G  (t,L,\mu)       \right| \, dt      \ll \frac{  T+  O_{\varepsilon} \left( A_{\varepsilon} (L) \right)   }{ \sqrt{W}}.     \]
Also integrating term by term we see that 
\[  \begin{split}
    \int_{T}^{2T}         L  m^{it}  \left( D_F (t,L,\mu)-    D_G  (t,L,\mu)       \right) \, dt  =  T \frac{c(m) }{ \sqrt{m} }\Lambda(m) \g \left( \frac{\log m}{L}\right) +  O\left(  \sum_{u \leq e^L} \frac{ | c(u)| \log u}{\sqrt{u}} \right) .    
  \end{split}   \]
  By similar calculations as done earlier in section \ref{background material}  we can check that  \[   \sum_{u \leq e^L} \frac{ | c(u)| \log u}{\sqrt{u}}= O_{\varepsilon} \left( A_{\varepsilon} (L) \right) ,      \]
  We conclude that 
  \begin{equation}{\label{12c}}
      \begin{split}
          \int_{t \in \tilde{\Lb} }  L   m^{it} \left( D_F (t,L,\mu)-    D_G  (t,L,\mu)       \right) \, dt  = T \frac{c(m) }{ \sqrt{m} }\Lambda(m) \g \left( \frac{\log m}{L}\right) + O \left(   \frac{T}{\sqrt{W}} \right) + O_{\varepsilon} \left(  A_{\varepsilon} (L)  \right).
      \end{split}
  \end{equation}
  Since $c(m)=0$ except on prime powers, we have that $\Lambda(m) \neq 0$.  Putting all information we have collected together (  (\ref{estimate on z 2}), (\ref{H,Z,D}), (\ref{12b}),  (\ref{12c}) ), 
  \begin{equation}{\label{main equation real}}
      \begin{split}
    1& \ll   \frac{c(m) }{ \sqrt{m} }\Lambda(m)=  \frac{ \int_{t \in \tilde{\Lb} }  L  m^{it}  \left( D_F (t,L,\mu)-    D_G  (t,L,\mu)       \right) \, dt + O \left(   \frac{T}{\sqrt{W}} \right) + O_{\epsilon} \left(  A_{\epsilon} (L) \right)  }{T \g \left( \frac{\log m}{L}\right)} \\
    &\ll \frac{1}{T \g \left(\log m/ L \right)}  \bigg(  T\left(  \mu^{-1/4}  + \C/T \right)  + E_{\mu} (2T) +E_{\mu} (T) \frac{T}{W} +  \mu^{-1/2} \Gamma(\mu +1) \left(  2/ L   \right)^ {\mu}\log(T)  \\ 
    &+   T \rho(T)  \int_{|y| \geq L/W \rho(T)} \left|       
 \h( y)    \right| \, dy  +    \frac{T}{\sqrt{W}}  + O_{\varepsilon} \left(  A_{\varepsilon} (L)      \right) \bigg). \end{split}
  \end{equation}
  By Stirling's formula we have $ g_{\mu} \left( \log m/L  \right) \sim   \left( 1-\left(\log m/L   \right)^2    \right)^{\mu-1/2}   $. Using the identity $1-x^2 \geq e^{-2x^2}$ for $|x|<1/2$ , together with the fact $L$ goes to infinity with $T$, for large enough $T$ we have that 
  $  e^{-2 \left(\mu-1/2 \right) \left(  \log m/L \right)^2 } \ll \g \left( \log m/L  \right)   $.\\
  
  We shall get a contradiction in (\ref{main equation real}) to complete the proof as then $c(m)=0$ for all $m$.
  To get a contradiction we shall make the following choices for $L, \mu, W, \varepsilon$. Let $W$ be a sufficiently large constant to be chosen in a moment and let $\varepsilon = 1/  2W^2  $ .   Finally, let $\mu= L/2W\rho(T)$ and $L= W^2 \rho(T) \log (\rho(T))$. Under these choices, for large enough $T$, $1 \ll  \g \left( \log m/L  \right)$. 
  \\
  
For the integral term, 
\begin{equation*}{\label{integral of mod h}}
\begin{split}
    \int_{|y| \geq L/W \rho(T)} \left|       
 \h( y)    \right|  \, dy &\leq \frac{  \Gamma(\mu+1)  }{\sqrt{\mu} } \int_{|y| \geq L/W \rho(T)} \left(\frac{2}{|y|} \right)^{\mu} \, dy 
 \ll \frac{1}{ \sqrt{\mu}}\frac{\Gamma(\mu+1)}{\mu-1} \left( \frac{2}{L/W\rho(T)} \right)^{\mu-1} .
\end{split}
\end{equation*}
Since  $\mu= L/2W\rho(T)$, using Stirling's formula, we obtain 
\begin{equation*}{\label{integral term}}
    \int_{|y| \geq L/W \rho(T)} \left|       
 \h( y)    \right|  \, dy \ll  e^{-\mu} .
\end{equation*}
 This estimate shows that 
$ \rho(T) \int_{|y| \geq L/W \rho(T)} \left|       
 \h( y)    \right|  \, dy \rightarrow 0   $ as $T \rightarrow \infty$ as desired. \\
 For the contribution of zeroes outside $(T-1,2T+1)$,  using the formula for $\mu$ in terms of $L$ and Stirling's formula, we get
 \begin{equation*}
     \begin{split}
      \frac{  \mu^{-1/2} \Gamma(\mu+1) \left(  2/ L   \right)^{\mu} \log(T)}{T}&= \frac{  \Gamma(\mu+1) \left(  1/\mu W\rho(T)   \right)^{\mu} \log(T)}{ \sqrt{\mu} T} \\
      & \ll \frac{  e^{-\mu} \log(T)  }{   (W \rho(T))^{\mu} T}.
            \end{split}
 \end{equation*}
 This converges to $0$ as $T \rightarrow \infty$ as desired. \\
 For the final term, under the choice of $ \varepsilon$, $ A_{\epsilon} (L) /T \rightarrow 0$ as $T \rightarrow \infty$. \\
 By Stirling's formula,  
 \[    \C \ll  \left( \frac{2}{e} \right)^{\mu} \mu^{\mu}.            \]
 Since $ {\mu}^{\mu}/ T \rightarrow 0$ as $T \rightarrow \infty$, we deduce that $ \C/ T \rightarrow 0 $ as $T \rightarrow \infty$. Hence $E_{\mu} (T)$ converges to $1$ as $T \rightarrow \infty$. 
Putting this all together, taking $T$ to be sufficiently large (\ref{main equation real}) implies that 
\[ 1 \ll \frac{1}{ \sqrt{W}} ,                  \]
which is clearly contradicted by choosing sufficiently large $W$, completing the proof. 
 
 \section{Proofs of lemmas}
 Recall our definitions of $\C$, $E_{\mu}(T)$  from Lemma \ref{lem for H}.  Also recall that $t \in [T,2T]$, $\mu \asymp \log \rho(T) $ , $L \asymp  \rho(T) \log \rho(T) $ .    We shall remove any $F$ or $G$ subscripts to make it more concise.
\begin{proof}[Proof of Lemma {\ref{lem for H}}]
 Looking at the integral formula for $H(t,L,\mu)$.  For the constant term in the integral, by the Fourier inversion formula  
  \begin{equation}{\label{integral of h_n (L(r-t))}} \int_{-\infty}^{\infty}        \h (L(r-t)) \, dr =  \frac{1}{L} \g(0)       .         \end{equation}
   Denote $ \left| \lambda_j (\frac{1}{2}+ir) + \mu_j      \right|  $ as $l_j (r)$. Since $l _j (r) $ is uniformly bounded away from $0$ for $r \in \mathbb{R}$ (by assumptions on $\lambda_j, \mu_j$) , Stirling's formula implies the digamma term is equal to 
\begin{equation}{\label{stirlings formuula for H}}
\begin{split}
   & \sum_{j=1}^{N} 2 \lambda_j \int_{-\infty}^{\infty} \h(L(r-t)) 
\Re   \log \left( \lambda_j (\frac{1}{2}+ir) + \mu_j            \right) \, dr +   \sum_{j=1}^{N} \int_{-\infty}^{\infty} O\left(\frac{1}{  l_j (r)} \right)  \left| \h(L(r-t))  \right| \, dr  \\  
&=   \sum_{j=1}^{N} 2 \lambda_j \int_{-\infty}^{\infty} \h(L(r-t)) 
   \log \left(  l_j(r)          \right) \, dr  
   +    \sum_{j=1}^{N} \int_{-\infty}^{\infty}  O\left(\frac{1}{  l_j (r)} \right)  \left| \h(L(r-t))  \right| \, dr .
   \end{split}
\end{equation}
For the error term, using the trivial estimate $ |\h(u)| \leq  \h(0)= 1/ \sqrt{\mu}$, and splitting the integral up into suitable  regions, we have 
\begin{equation*}
\begin{split}
     \int_{-\infty}^{\infty}  \frac{1}{  l_j (r)} \left| \h(L(r-t))  \right| \, dr & \ll 
          \frac{1}{L}   \left( \frac{1}{\mu^{1/4}}  + \int_{|u|>\mu^{1/4}}  \frac{1}{  l_j (\frac{u}{L} +t)}    \left| \h(u)  \right| \, du                         
                      \right).
     \end{split}
\end{equation*}
When  $ u > \mu^{1/4} $, $ 1/   l_j (\frac{u}{L} +t)\ll 1/ \left( \frac{u}{L} +t    \right) \leq 1/t \leq 1/T    $, so 
 \[  \int_{u>\mu^{1/4}}  \frac{1}{  l_j (\frac{u}{L} +t)}    \left| \h(u)  \right| \,du  \ll \frac{\C}{ T \left( \mu^{1/4} \right)^{\mu-1} (\mu-1) }  \ll \frac{\C}{T} .  \]
 We still need to consider the integral when $ u\leq - \mu^{1/4}$. We would like to use the estimate  \\ $1/ l_j (\frac{u}{L} +t)\ll 1/ \left|\frac{u}{L} +t    \right|$, however it is not useful in a neighbourhood of $ -tL$ (which  definitely lies in this region by our choice of $\mu$, $L$)  . In this region we use $1/l_j(\frac{u}{L}+t)= O(1)$. Let $a= LT/2 $. We have the estimate 
 \[  \int_{-tL -  a}^{-tL + a}    \frac{1}{  l_j (\frac{u}{L} +t)}    \left| \h(u)  \right| \,du     \ll  \frac{\C}{ (\mu-1) \left( Lt-a \right)^{\mu-1}} \ll \frac{\C}{(\mu-1) a^{\mu-1}}\ll  \frac{\C}{T} .            \] 
 On $ ( -Lt+ a , - \mu^{1/4} )   $, we have  $ 1 / \left|\frac{u}{L} +t    \right| \ll 1/T$, so 
 \[      \int_{-tL + a}^{ - \mu^{1/4}}    \frac{1}{  l_j (\frac{u}{L} +t)}    \left| \h(u)  \right| \,du  
 \ll \frac{\C}{ T (\mu^{1/4} )^{ \mu-1} (\mu-1)}  \ll   \frac{\C}{T}  .       \]
 Finally, by the same analysis,
 \[     \int_{-\infty}^{ -tL-a }    \frac{1}{  l_j (\frac{u}{L} +t)}    \left| \h(u)  \right| \, du \ll \frac{ \C}{ T (\mu-1) \left( tL+a     \right)^{\mu-1} \ll       }   \ll \frac{ \C}{ T (\mu-1) \left( 3TL/2     \right)^{\mu-1}      }  \ll  \frac{\C}{T}  . \]
 Putting all of our analysis together, we see that 
 \begin{equation}{\label{ O(1/r) error term for h_n}}
 \begin{split}
     \int_{-\infty}^{\infty}  \frac{1}{  l_j (r)} \left| \h(L(r-t))  \right| \, dr &\ll \frac{1}{L}   \left(   \frac{1}{\mu^{1/4}} + \frac{\C}{T} \right).
 \end{split}
     \end{equation}

     We have dealt with the error term in (\ref{stirlings formuula for H}). Now for the first term, 
notice that $l_j(T) \sim \lambda_j T  $, so 
when $r \in [T- \sqrt{T}, 2T+ \sqrt{T}]$ we have $\log l_j(r) = \log T + O(1)  $.  
Fixing $j$ and recalling $t \in [T,2T]$, we have 
\begin{equation}{\label{integral of h times l}}
    \begin{split}
        \int_{-\infty}^{\infty} \h(L(r-t)) 
   \log \left(  l_j(r)        \right) dr& =    \int_{t- \sqrt{T}}^{t+ \sqrt{T}} \h(L(r-t)) \left(  \log T +O(1)   \right) \, dr \\
   &+ \int_{ | r-t| > \sqrt{T}} \h  ( L(r-t))   \log l_j(r) \, dr.
   \end{split}
   \end{equation}
   Notice  that    when $|u|> L \sqrt{T}$, we have $\log l_j( \frac{u}{L}+t ) \ll \log(|u|)$ since $ \sqrt{t} \ll u$. Integrating by parts we have  
   \[   \int_{  |u|  > L \sqrt{T}} \h  ( u)   \log l_j \left( \frac{u}{L}+t \right) \, du \ll \frac{\C}{ (\mu-1) \left(  L \sqrt{T} \right)^{\mu-1} } \left(  \log (L \sqrt{T}) +\frac{1}{\mu-1}     \right)  .  \] This implies 
   \[    \int_{ | r-t| > \sqrt{T}} \h  ( L(r-t))   \log l_j(r) \, dr \ll   \frac{\C}{LT}.    \]
   For the error term in the first integral in (\ref{integral of h times l}), 
  \[  \int_{t- \sqrt{T}}^{t+ \sqrt{T}} \h(L(r-t)) O(1)   \, dr   \ll \frac{\C}{ L (\mu-1) (L \sqrt{T}) ^{\mu-1}} \ll     \frac{\C}{LT}       .         \]
   This also implies 
   \[   \frac{\log (T)}{2 \pi}\int_{t-\sqrt{T}}^{t+ \sqrt{T}} \h(L(r-t)) \,dr =  \frac{\log(T)}{L} \left( \g (0) + O\left(  \frac{\C}{(\mu-1) (L \sqrt{T}) ^{\mu-1}}     \right)     \right)                     . \] 
  
   Putting this all together,  
   \begin{equation}{\label{integral of h times log l}}
   \begin{split}
       \frac{1}{2 \pi} \int_{-\infty}^{\infty} \h(L(r-t)) 
   \log \left(  l_j(r)        \right) \, dr&= \frac{\log (T)}{L}  \g (0)   
   +  \frac{1}{L} O \left(  \frac{\C}{T}  
   \right).
   \end{split}
     \end{equation}
     Using the fact $\g (0)= O(1)$ , (\ref{integral of h_n (L(r-t))}), (\ref{stirlings formuula for H}) , (\ref{ O(1/r) error term for h_n}) and (\ref{integral of h times log l}) show that  
     \begin{equation*}
     \begin{split}
         H_F (t,L,\mu) =  \g (0) \frac{ d_F \log(T)}{L} + &\frac{1}{L} O \left(   E_{\mu}(T) \right).
      \end{split}
     \end{equation*} 
   \end{proof}
   \begin{proof}[Proof of Lemma \ref{lem for derivative term}]
   Integrating by parts and using $d_F=d_G$, which follows from \eqref{primes of E} and \cite{Sound},  Lemma \ref{lem for H} implies  that
   \[  \begin{split}
       \int_{T}^{2T} L m^{it} \left( H_F (t,L,\mu) - H_G (t,L,\mu) \right) \, dt & =        - \frac{1}{ i \log m} \int_{T}^{2T} L \frac{d}{dt} \left( H_F (t,L,\mu) - H_G (t,L,\mu)  \right)   m^{it} \, dt \\
       &+ O \left( \E (2T) +\E(T)  \right).  \\
   \end{split}     \]
   Hence to complete the proof we just need to estimate $L  \frac{d}{dt} \left( H_F (t,L,\mu) - H_G (t,L,\mu)  \right) $. Differentiating under the integral expression of $ H_F (t,L,\mu) - H_G (t,L,\mu) $ and  integrating by parts again, we get 
   \begin{equation*}
   \begin{split}
      &  L  \frac{d}{dt} \left( H_F (t,L,\mu) \right)=  \\ & -  \left(  2 \log Q_F + \frac{\Gamma'_F}{\Gamma_F} \left( \frac{1}{2}+ir \right) + \frac{\overline{\Gamma'_F}}{\overline{\Gamma_F}} \left( \frac{1}{2}- ir \right)     \right) L \h (L(r-t))        \Bigg|_{r=-\infty}^{\infty}  \\
  & +  iL \int_{-\infty}^{\infty} \h (L(r-t)) \left(   \left(  \frac{\Gamma'_F}{\Gamma_F}\right)'  \left( \frac{1}{2}+ir \right)     -  \left(  \frac{\overline{\Gamma'_F}}{\overline{\Gamma_F}} \right)' \left( \frac{1}{2}- ir \right) \right)    \,  dr .
   \end{split}     \end{equation*}
   By Stirling's formula and our asymptotic estimate on $\h$, the first term is  zero. By Stirling's formula for the trigamma function, the trigamma terms in the integral are  $O(1/l_j(r))$ . Using (\ref{ O(1/r) error term for h_n}), we have that the integral term is $ \ll  \mu^{-1/4} + \C/T $, 
     thus completing our proof. 
   
   \end{proof}
% \begin{rem}
 %Let $1>\epsilon>0$ and let $a=1- \epsilon>0$. Want an estimate for $\sum_{p} logp/p^{\epsilon}$. We shall use Mertens estimate and summation by parts as follows. 
 %\begin{equation*}
 %\begin{split}
  %    \sum_{p \leq x} \frac{logp}{p^{\epsilon}}&= \sum_{p \leq x} \frac{log p}{p} p^{a}  \\
   %   &= \sum_{p \leq x}\frac{log p}{p} \left( \int_{2}^{p} a t^{-\epsilon} dt + 2^{a}   \right)   \\
    %  &= \int_{2}^{x} at^{-\epsilon} \sum_{t \leq p \leq x} \frac{log p}{p} dt + log (x) 2^{a} + O_{\epsilon} (1) 
% \end{split}
 %\end{equation*}
 %Noticing that $ \sum_{t \leq p \leq x} logp/p = log(x)- log(t) +O(1)       $
 %and $ \int_{2}^{x} a t^{-\epsilon} log(t) dt= log (x) x^{a}-x^{a}/a +O_{\epsilon} (1)$, we get 
 %\[  \sum_{p \leq x} \frac{logp}{p^{\epsilon}}= \frac{x^a}{a} + O(x^{a}) + log(x) 2^{a} + O_{\epsilon} (1)                     \]
 
% \end{rem}

\providecommand{\bysame}{\leavevmode\hbox to3em{\hrulefill}\thinspace}
\providecommand{\MR}{\relax\ifhmode\unskip\space\fi MR }
\providecommand{\MRhref}[2]{%
  \href{http://www.ams.org/mathscinet-getitem?mr=#1}{#2}
}
\providecommand{\href}[2]{#2}

%\bibliographystyle{amsplain}
%\bibliography{ms}

  \textbf{School of Mathematics, University of Bristol, Bristol, BS8 1UG, United Kingdom }
  \\
 \textbf{Email address: michaelfarmer868@gmail.com}

\end{document}